\documentclass[reqno]{amsart}
\usepackage{amssymb}
\usepackage{xcolor}
\usepackage[ansinew]{inputenc}
\usepackage[T1]{fontenc}
\usepackage[english]{babel}
\usepackage{amsthm}
\usepackage{amsfonts}         
\usepackage{amsmath}
\usepackage{comment}
\usepackage{amssymb}
\usepackage{empheq}
\usepackage{graphicx}
\usepackage{url} 
\usepackage{framed}
\usepackage{amsmath}
\usepackage{amssymb}
\usepackage{latexsym}
\usepackage{enumerate}
\usepackage{geometry}

\newcommand{\p}{{\partial}} 
\newcommand{\R}{{\mathbb R}} 
\newcommand{\N}{{\mathbb N}}

\newcommand{\D}{{\mathcal D}} 

\def \beq {\begin {eqnarray}}
\def \eeq {\end {eqnarray}}
\def \ba {\begin {eqnarray*} }
\def \ea {\end {eqnarray*} }

\newtheorem{definition}{Definition}[section] 
\newtheorem{theorem}[definition]{Theorem} 
\newtheorem{lemma}[definition]{Lemma} 
 
\newtheorem{proposition}[definition]{Proposition}

\theoremstyle{plain}

\title[Travel time difference]{Inverse problem of Travel time difference functions on compact Riemannian manifold with boundary}

\begin{document}
\author[M. V. de Hoop]{Maarten V.  de Hoop $^{\diamond}$}


\author[T. Saksala]{Teemu Saksala $^{\diamond,\: \ast}$}

\let\thefootnote\relax\footnote{ 
$^\diamond$  Department of Computational and Applied mathematics, Rice University, USA
\\ 
$^\ast$ \textbf{teemu.saksala@rice.edu}}

\begin{abstract}
We show that the travel time difference functions, measured on the boundary, determine a compact Riemannian manifold with smooth boundary up to Riemannian isometry, if boundary satisfies a certain visibility condition. This corresponds with the inverse microseismicity problem. 


The novelty of our paper is a new type of a proof and a weaker assumption for the boundary than it has been presented in the literature before. We also construct an explicit smooth atlas from the travel time difference functions. 
\end{abstract}

\maketitle

\section{Introduction}
\label{Se:motivation}
Let $(N,g)$ be a complete, connected smooth Riemannian manifold. We split the manifold into two parts that are a closed set $M$, with non-empty interior, and the closure of the exterior $F:=\overline{N \setminus M}$. We assume that the boundary $\p M$ of $M$ is a smooth co-dimension one manifold. The set $F$ is the known observation domain and $M$ is the object of interest, for instance Earth. The Riemannian metric $g$ can be seen as a proxy of the material parameters of $M$. 

For any $p,q \in N$ we denote by $d_N(p,q)$ the length of a distance minimizing geodesic of $(N,g)$ that connects $p$ to $q$. We assume that the wave speed in $F$ is much slower than in $M$. Especially if $\p M$ is strictly convex, we may assume that  distance minimizing geodesics of $(N,g)$ connecting $p$ to $q$ stay inside $M$, if $p,q\in M$. This implies
\begin{equation}
\label{eq:exterior_dist}
 d_M(p,q)=d_{N}(p,q),  \quad p,q \in M,
\end{equation}
where  $d_M(p,q)$ is the distance from $p$ to $q$ in $M$, that is given as the infimum of lengths of curves from $p$ to $q$ that stay in $M$. For a while we assume that \eqref{eq:exterior_dist} holds and we denote $d_M=d_g$.

Suppose that there exists a Dirac point source $(p,s)\in M \times \R$ of a Riemannian wave equation, with zero Cauchy data. It follows from \cite{duistermaat1972fourier} and \cite{greenleaf1993recovering} that the singularities emitted from $(p,s)$ propagate along the geodesics of $(N,g)$ (see for instance \cite{LaSa} for more details).
For every $z \in  \p M$ we define the \textit{arrival time} $\mathcal{T}_{p,s}(z)$ to be the infimum of times when a spherical wave emitted form $(p,s)$ is observed at $z$.  Hence $\mathcal T_{p,s}(z)=d_g(p,z)+s$, and the  \textit{travel time difference  function} satisfy an equation
\begin{equation}
D_p(z_1,z_2):=d_g(p,z_1)-d_g(p,z_2) =\mathcal T_{p,s}(z_1)-\mathcal T_{p,s}(z_2), \quad z_1,z_2 \in \p M.
\label{eq:Relation of wave data and DDD}
\end{equation}
The important property of this function is that it is given as the difference of the arrival times.
The knowledge of the emission time $s$ or the origin remains unknown, but the function $D_p$ can be determined without  knowledge on $s$.  
\color{black} 
This paper is devoted to the study of the inverse problem of travel time difference functions. This problem can be formulated as follows. Does the collection 
$$
\{D_p:p \in M^{int}\},
$$
determine the Riemannian manifold $(M,g)$ up to isometry?  


\bigskip
Now we give our problem setting. Let $(M,g)$ be a compact connected $n$--dimensional Riemannian manifold with smooth boundary $\p M$. 
Since $M$ is compact for any points $p,q \in M$ there exists a distance minimizing $C^1$--smooth curve $c$ from $p$ to $q$, see \cite{alexander1981geodesics}. Moreover for any $t_0 \in [0,d_g(p,q)]$ such that point $\gamma(t_0) $ is an interior point of $M$ there exists $\epsilon>0$ such that $c:(t_0-\epsilon, t_0 +\epsilon)$ is a geodesic. We denote the collection of all interior points of $M$ by $M^{int}$. 
We use the notation $SM$ for the unit sphere bundle of $(M,g)$. Therefore each $(p,v) \in SM$ determines the  unique maximal unit speed geodesic $\gamma_{p,v}$ of $(M,g)$. 

For any $p \in M$ we define the corresponding \textit{travel time difference function}.
\begin{equation}
\label{eq:DDF}
D_p:\p M \times \p M \to \R, \quad D_p(z_1,z_2):=d_g(p,z_1)-d_g(p,z_2).
\end{equation}
Notice that the function $D_p$ is continuous. We assume that the following \textit{travel time difference data} 
\begin{equation}
\label{eq:data}
(\p M, \: \{D_p: \: p \in M^{int}\}),
\end{equation}
is given. That is we assume, that the  $(n-1)$--dimensional smooth manifold  $\p M$ without boundary and the collection of functions $\{D_p:\p M \times \p M \to \R \: |\: p \in M^{int}\}$ are given. We  emphasize that a priori the points $p$ related to $D_p$ are unknown.

\medskip
The aim of this paper is to prove that travel time difference data determine $(M,g)$ up to isometry. Before stating our main theorem, we describe an additional geometric property for $\p M$ under which we can prove the uniqueness of the inverse problem. 

Let $(N,G)$ be any smooth closed Riemannian manifold that extends $(M,g)$, such that  $g=G|_{M}$.
We  use the notation 
\[
\ell(x,v):=\inf \{t > 0: \gamma_{x,v}(t) \in N \setminus M\}, \quad (x,v)\in SM.
\]
Thus the domain of definition for $\gamma_{x,v}$ is $[-\ell(x,-v),\ell(x,v)]$. Moreover by Lemma 1 of \cite{stefanov2009}, $\ell(x,v)$ is independent of the extension.  We note that $\gamma_{x,v}$ may intersect the boundary tangentially in many points.


\begin{definition}
\label{eq:SU-cond-2}
We say that $(M,g)$ satisfies \textit{the visibility} condition, if the following holds: For every $z \in \p M$ there exists $(z,\eta) \in \p S M, \hbox{ such that } \ell(z,\eta) < \infty.$  Geodesic $\gamma_{z,\eta}: [0,\ell(z,\eta)] \to M$ is a distance minimizer and $\gamma_{z,\eta}(\ell(z,\eta))$ is not a cut point to $z$, $\dot{\gamma}_{z,\eta}(\ell(z,\eta))$ is tranversal to $\p M$ and  $\gamma_{z,\eta}((0,\ell(z,\eta))) \subset M^{int}$. 
\end{definition}



\medskip
Next, we formulate our main Theorem. Let $(M_1,g_1)$ and $(M_2,g_2)$ be two smooth compact Riemannian manifolds with smooth boundaries $\p M_1$ and $\p M_2$. 
\begin{definition}
\label{de:TTDD_agree}
We say that the travel time difference data of $(M_1,g_1)$ and $(M_2,g_2)$ coincide, if there exists a diffeomorphism  $\phi:\p M_1 \to \p M_2$ such that
\begin{equation}
\label{eq:equivalent_data}
\{D_p(\phi^{-1}(\cdot),\phi^{-1}(\cdot)): p \in M_1^{int}\}=\{D_q: q \in M_2^{int}\}.
\end{equation}
\end{definition}

Then.
\begin{theorem}
\label{th:main}
Let $(M_i,g_i),\:  i=1,2$ be  compact, connected $n$--dimensional Riemannian manifolds with smooth boundaries $\p M_i$. Suppose that $(M_1,g_1)$ satisfy the visibility condition \ref{eq:SU-cond-2}. If  the travel time difference data of $(M_1,g_1)$ and $(M_2,g_2)$ coincide, then there exists a Riemannian isometry $\Psi:(M_1,g_1) \to (M_2,g_2)$ such that the restriction of $\Psi$ on $\p M_1$ coincides with $\phi$.
\end{theorem}
\color{black}

While preparing this paper for submission, the authors became aware                 
that S. Ivanov very recently posted a preprint \cite{ivanov2018distance} on ArXiv with a result (Proposition 7.3.) related                
to the result presented here. Indeed, he proved a similar result for                
complete manifolds with boundary under the assumption that the                      
boundary is nowhere concave. On the other hand by the proof of Lemma \ref{Le:jet}, the  claim of  Proposition 7.3. in \cite{ivanov2018distance} holds if the nowhere concave boundary condition is replaced with the visibility condition.  

We give a different proof for Theorem \ref{th:main} (see Section \ref{Se:outline} for the outline of our proof) compared to one given in \cite{ivanov2018distance}. The proof given in \cite{ivanov2018distance} is based on distance comparison inequalities implied by Toponogov's theorem and minimizing geodesic extension property. The latter property provides a lower bound on the length of a minimizing extension
of a geodesic beyond a non-cut point in terms of the length of a minimizing
extension beyond the other endpoint.

\medskip
We end this section by comparing the visibility condition to the nowhere concave boundary condition. Recall that the boundary $\p M$ of Riemannian manifold $(M,g)$ is nowhere concave, if for every $z \in \p M$ the second fundamental form of $\p M$ at $z$, with respect to the inward-pointing normal vector, has at least one positive eigenvalue. If $\p M$ is nowhere concave then by the proof of Proposition 3.4. of \cite{zhou2012recovery} and Section 4.1. of \cite{sharafutdinov2012integral} it holds that $(M,g)$ satisfies the visibility condition. Notice that an annulus, contained in  Euclidean plane, satisfies the visibility condition, but not the nowhere concave boundary condition. Therefore the visibility condition is more general of these two.

Finally we will give an example of such geometry that does not satisfy either of these boundary conditions. Let $M\subset S^2$ be a spherical cap larger than the half--sphere. If $g$ is the round metric on $M$, then $(M,g)$ does not satisfy the visibility condition, since any $g$--distance minimizing curve between boundary points lies in $\p M$ and therefore it is not a geodesic of $S^2$. In this case $\p M$ is not either nowhere concave.
\subsection*{Background}

\subsubsection{Four geometric inverse problems related to the Riemannian wave equation}

In this section we assume that $N,\: M, \: F$ and $g$ are as in Section \ref{Se:motivation}. There are four different data sets that are all related to Riemannian wave equation with the Dirac point source $(p,s)\in M\times \R$ and zero Cauchy data. 

The inverse problem of travel time functions have been considered in \cite{Katchalov2001,kurylev1997multidimensional}. The authors study the properties of the map $\mathcal{R}:M \to C(\p M)$, in which a point $p \in M$ is mapped into the corresponding travel time function $r_p:\p M \to \R$, given by the formula
$$
r_p(z)=d_g(p,z), \quad z \in \p M.
$$ 
The authors show that the data $(\p M, \{r_p: p \in M\})$ determine a manifold $(M,g)$ up to isometry. They use the map $\mathcal R$ to construct an isometric copy of $M$ in $C(\p M)$. They don't pose any restrictions to the geometry.  

In \cite{LaSa} the authors prove a result related to Theorem \ref{th:main}. In this paper it is assumed that the travel time difference function is given in the \textit{observation} set $F$ with non-empty interior 
$$
D_p:F \times F \to \R.
$$
In addition they assume that the Riemannian structure of $(F,g)$ is known. The proof of the main theorem in \cite{LaSa} is very similar to the proof of Theorem \ref{th:main} presented in this paper and we will often refer to it for the details that are not presented in this paper.

In \cite{ivanov2018distance} S. Ivanov extends the result of \cite{LaSa} in the following set up. Let $M$ be any complete, connected Riemannian manifold without boundary. Let $F,U \subset M$ be open. If the topology and differential structure of the observation domain $F$ and $D_p, \: p \in U$ are given then these data determine the geometry of the domain $(U,g_U)$ uniquely up to a Riemannian isometry. The sets $U$ and $F$ can be faraway from each other, which is not the case in \cite{LaSa} where it is assumed that $U=M$.  Furthermore S. Ivanov proves that the determination of $(M,g)$ from travel time difference functions $D_p$ is stable, if the underlying manifold has a priori bounds on its diameter, curvature, and injectivity radius. In \cite{ivanov2018distance} also a similar result to our Theorem \ref{th:main} is provided for complete manifolds with nowhere concave boundary.

The inverse problem related to the set of exit directions 
\[
\Sigma_p=\{(\gamma_{p,v}(\ell(p,v)),\dot \gamma_{p,v}(\ell(p,v)))\in \p SM: v \in S_pM\}
\]
of geodesics emitted from $p$ has been studied in \cite{lassas2018reconstruction}. Let 
$$
I(g,w,z,l):= \hbox{ number of $g$--geodesics of lenght $l $ connecting $w$ to $z$}, \quad w,z \in N, \: l >0
$$ 
The authors show that, if $(N,g)$ is a closed manifold such that 
\begin{equation}
\label{eq:generic}
\sup_{w,z,\ell} I(g,w,z,l) <\infty,
\end{equation}
$M$ is non-trapping and $\p M$ is strictly convex, then the collection of exiting directions
$$
\{\Sigma_p \subset \p TM : p \in M^{int}\}
$$
determine the manifold $(M,g)$ up to isometry. Assumption \eqref{eq:generic} is needed to show that each set $\Sigma_p$ is produced by the unique $p\in M$. To our understanding, it is not known, if \eqref{eq:generic} follows from the convexity of the boundary and non-trapping properties. On the other hand in \cite{kupka2006focal} it is shown that \eqref{eq:generic} is a generic property in the space of all Riemannian metrics of $N$. 

The final data set is related to a \textit{generalized sphere} of radius $r>0$, that is given by formula
\[
S(p,r)=:\{\exp_p(v): v\in T_pM,\; \|v\|_g=r,\; \hbox{$\exp_p$ is not singular at $v$}\}.
\]
In \cite{deHoop1} the authors show that the spherical surface  data
$$
\{S(q,r)\cap F: q  \in M, \: r >0 \}
$$
determine the universal cover space of $N$. If a generalized sphere $S(p,r)$ is given the authors show that there exists a specific coordinate structure in a neighborhood of any maximal normal geodesic to $S(p,r)$ such that in these coordinates metric tensor $g$ can can be determined. However this does not determine $g$ globally. The authors provide an example of two different metric tensors which produce the same spherical surface data. 

\subsubsection{Microseismicity}
In this paper the results in \cite{LaSa} are adapted, in a fundamental way, to data available from actual seismic surveys. The point sources are microseismic events detected in dense arrays at Earth's surface. In our theorem we show that the data determine the metric up to change of coordinates. This implies that one can locate  the closest surface point and to determine the corresponding travel time to each event.  

For the following  we assume that $M \subset \R^m$ and $p \in M^{int}$. Recall that the arrival time function is  $\mathcal{T}_{p,s}(z)=d_g(p,z)+s$, where $z\in \p M$ is a receiver point and $s\in \R$ is the emission time. Since $\mathcal{T}_{p,s}(z)$ is a highly non-linear function of $p$ it is traditional in seismological literature to study the linearization of $ \mathcal{T}_{p,s}(z)$  \cite{waldhauser2000double}. Let $p_0 \in M^{int}$ be a master event i.e. an event for which $d_g(p_0,z)$ is known and $d_g(\cdot,z)$ is $C^1$--smooth near $p_0$. By the Taylor series of $d_g(\cdot,z)$ we have that the linearization
\[
r^z_p:=\nabla d_g(\cdot,z)\bigg|_{p_0}\cdot (p-p_0) \approx d_g(p,z)-d_g(p_0,z),
\]
where, $\nabla$ is the Euclidean gradient and $p$ is close to $p_0$. The \textit{double difference distance function}  is $r^z_p-r^z_q$. This function is the difference of differential distances between a (receiver) point $z$ at the boundary, and two source points $p,q$  in the interior -- in which the metric is unknown -- of a manifold.   The goal is to use this data to determine travel time $d_g(p,z)$ of the second event and to locate the relative distance $d_g(p,p_0)$ of the second event to the master event.

The event location with this method is known as the DD earthquake location algorithm presented in \cite{waldhauser2000double}. This method assumes a flat earth model and is appropriate for local scale
problems.  In contrast to seismological literature we  measure the difference of the arrival  times $\mathcal{T}_{p,s}(z),\mathcal{T}_{p,s}(w)$ of the given event $p\in M^{int}$ to two  receivers $z,w \in \p M$. For our theorem it is not necessary to linearize the arrival times. 

The travel time difference function, given in \eqref{eq:DDF}, is closer related to applications in exploration seismology with the purpose of locating microseismic events Grechka \textit{et al.} \cite{grechka2015relative}. In this paper the authors assume that the travel time to the receivers and location of the master event is known. Notice that our result do not recover the locations of the events in Cartesian coordinates.


In global seismology, the idea to decouple the earthquake doublets, that is two different events that are close to each other and produce nearly indentical waveform, to locate the events was introduced by Poupinet \textit{et                                               
al.} \cite{PoupinetEF-1984}. Zhang \& Thurber \cite{ZhangT-2006, ZhangT-2003} extend the              
double difference location method of Waldhauser \&                                  
Ellsworth \cite{waldhauser2000double} with an attempt to simultaneously solve for               
both velocity structure and seismic event locations. They develop  a
regional DD seismic tomography methods that deal effectively with
discontinuous velocity structures without knowing them a priori. Their methods also take Earths curvature into account.

\color{black}

\section{Proof of the Main theorem}
In this section we prove Theorem \ref{th:main}. Whenever it is not necessary to distinguish manifolds $M_1$ and $M_2$ from one other we drop the subindices. In these cases we work with the data \eqref{eq:data}.

\subsection{Outline of the proof of the Main theorem}
\label{Se:outline}

The proof consists of three steps. First we use the data \eqref{eq:data} to construct a mapping $\D$ from points of $M$ to continuous functions on $\p M \times \p M$. We show that this mapping is a topological embedding. Then we use the diffeomorphism $\phi:\p M_1 \to \p M_2$ and \eqref{eq:equivalent_data} to construct a homeomorphism $\Psi:M_1 \to M_2$ as in Theorem \ref{th:main} (see \eqref{eq:map_psi} for the definition). In second part we show that this mapping is a diffeomorphism. We prove the existence of such local coordinate maps that are determined by \eqref{eq:data}. In the third part we first prove that the data \eqref{eq:data} determine the images of geodesic segments that come to the boundary $\p M$. Finally we use this information to prove the uniqueness of Riemannian structure.

The outline of the proof of the main theorem is similar to the proof of the main theorem of \cite{LaSa}. The proof presented in this paper contains two key differences to the earlier result. The first one is the construction of the boundary coordinate system, in the beginning of Section \ref{Se:smooth}. The determination of the boundary defining function (see \eqref{eq:func_f_p} and \eqref{eq:boundary_def_func}), only from the data \eqref{eq:data}, has not been presented in the literature before. The second difference, that is considered in the beginning of Section \ref{Se:Riemannian}, is related to the construction of metric tensor from the data \eqref{eq:data}. 
In order to use the similar techniques as in \cite{LaSa}, to prove that the metrics $g_1$ and $\Psi^{\ast}g_2$ coincide, we need to prove that the data \eqref{eq:data} determine the full Taylor expansion of the metric tensor on $\p M$ in boundary normal coordinates. This makes it possible to extend $M_1$ to a closed manifold $N$ given with two smooth metric tensors $G$ and $\widetilde G$ that coincide in $F:=\overline{N\setminus M_1}$, $G|_{M_1}=g_1$ and $\widetilde G|_{M_1}=\Psi^{\ast}g_2$. Since we don't assume $\p M$ to be strictly convex, we will need also to show that the travel time difference functions $D_p:F\times F \to \R, \: p\in N$ of $(N,G)$ and  $(N,\widetilde G)$ coincide.  For this last step we use the proof of the Proposition 7.3 of \cite{ivanov2018distance} by S. Ivanov. The visibility condition of the Definition \ref{eq:SU-cond-2} is needed to tackle these problems. 

\color{black}
\subsection{Topology}
We start first extending the data to the boundary. If $p,w \in \p M$ then by the triangle inequality it holds that
\begin{equation}
\label{eq:boundary_distance}
d_g(p,w)=\sup_{q \in M^{int}}D_q(p,w).
\end{equation}
Thus data \eqref{eq:data} determine $d_g:\p M \times \p M \to \R$ and the extended data
\begin{equation}
\label{eq:data_full}
(\p M, \{D_p: \: p \in M\}).
\end{equation}
Our first Lemma is

\begin{lemma}
Let $(M_i,g_i),\:  i=1,2$ be compact $n$--dimensional Riemannian manifolds with smooth boundaries $\p M_i$. If  the travel time difference data of $(M_1,g_1)$ and $(M_2,g_2)$ coincide, then
\begin{equation}
\{D_p(\phi^{-1}(\cdot),\phi^{-1}(\cdot)): p \in M_1\}=\{D_q: q \in M_2\}. \label{eq:equivalent_full_data}
\end{equation}
\end{lemma}
\begin{proof}
From \eqref{eq:equivalent_data} and \eqref{eq:boundary_distance} it follows that 
\begin{equation}
\label{eq:boundary_dist_agree}
d_1(\phi^{-1}(p),\phi^{-1}(q))=d_2(p,q), \quad p,q \in \p M_2.
\end{equation}
Here, $d_i$ is the distance function of $g_i$ for $i \in \{1,2\}$. Therefore \eqref{eq:equivalent_full_data} holds.
\end{proof}
We study the properties of the mapping 
$$
\D:M \to C(\p M \times \p M), \quad \D(p)=D_p,
$$
where the target space is equipped with the $L^\infty$--norm.

\begin{lemma}
\label{pr:topology}
The mapping $\D$ is a topological embedding.
\end{lemma}
\begin{proof}
Using triangle inequality it is easy to see that $\D$ is $2$--Lipschitz.

\medskip
Next we prove that $\mathcal{D}$ is one-to-one. To show this, assume that $x,y \in M$ are such that  $D_x=D_y$. We first show that this implies that  the set $\{z_x\}$ of closest boundary points of $x$ coincides with the set $\{z_y\}$ of closest boundary points of $y$.  Let $w \in \p M$ and define
\begin{equation}
\label{eq:func_f_w}
f_{x,w}:\p M \to \R, \quad f_{x,w}(z):=D_x(z,w).
\end{equation}
Then $\{z_x\}$ is the set of minimizers of function $f_{x,w}$. Since $f_{x,w}=f_{y,w}$, we have proven that $\{z_x\}=\{z_y\}$. We also use the function $f_{x,w}$ later when we construct a boundary defining function.

Let $z_0 \in \{z_p\}$ and denote  $s_x=d_g(x,z_0)$ and $s_y=d_g(y,z_0)$. Without loss of generality, we can assume that  $s_x \leq s_y$. Let $\nu$ be the inward pointing unit normal vector field to $\p M$. Then $\gamma_{z_0,\nu}$ is the distance minimizing geodesic from $\p M$ to $x$ and $y$. Moreover 
\begin{equation}
\label{eq:dist_from_x_to_y}
x=\gamma_{z_0,\nu}(s_x), \:  y=\gamma_{z_0,\nu}(s_y) \hbox{ and } d(x,y)=s_y-s_x. 
\end{equation}
If $z \in \p M\setminus \{z_0\}$  is close to $z_0$, the distance minimizing geodesic $\gamma_x$ from $z$ to $x$ is not the same geodesic as $\gamma_{z_0,\nu}$, that is, the angle  $\beta$ of the curves $\gamma_x$ and $\gamma_{z_0,\nu}$ at the point $x$ is strictly between $0$ and $\pi$. Let $\gamma_y$ be a distance minimizing geodesic from $y$ to $z$. We note that $D_x(z,z_0)=D_y(z,z_0)$ and \eqref{eq:dist_from_x_to_y} yields
$$
\mathcal{L}(\gamma_y)=d(y,z)=d(y,x)+d(x,z)=\mathcal{L}(\gamma_{z_x,\nu}|_{[s_x,s_y]})+\mathcal{L}(\gamma_x).
$$
Thus the union $\mu$ of the curves $\gamma_{z_x,\nu}([s_x,s_y])$ and $ \gamma_x$ is a distance minimising curve from $z$ to $y$, and hence it is a geodesic. However, as the angle $\beta$, defined above, is strictly between $0$ and $\pi$, the curve  $\mu$ is not smooth at $x$, and hence it is not possible that $\mu$ is a geodesic unless $x=y$. Thus $x$ and $y$  have to be equal.

\medskip Since $M$ is compact and we just proved that $\D$ is continuous and one--to--one, we have that mapping $\D$ is closed. Thus the claim is proven.
\end{proof}
Since the mapping $\phi$, given by Definition \ref{de:TTDD_agree}, is a diffeomorphism the mapping
$$
\Phi:C(\p M_1 \times \p M_1) \to C(\p M_2 \times \p M_2), \quad \Phi(F)=F(\phi^{-1}(\cdot),\phi^{-1}(\cdot))
$$
is an isometry. Let $\D_i, \: i\in \{1,2\}$ be as $\D$ on $(M_i,g_i)$. Now we are ready to define the mapping
\begin{equation}
\label{eq:map_psi}
\Psi:M_1 \to M_2, \quad \Psi = \D_2^{-1} \circ \Phi\circ \D_1.
\end{equation}
\begin{proposition}
\label{th:topology}
Let $(M_i,g_i),\:  i=1,2$ be compact $n$--dimensional Riemannian manifolds with smooth boundaries $\p M_i$. If  the travel time difference data of $(M_1,g_1)$ and $(M_2,g_2)$ coincide, then the mapping $\Psi$ given by \eqref{eq:map_psi} is a homeomorphism such that the restriction of $\Psi$ on $\p M_1$ coincides with $\phi$. 
\end{proposition}
\begin{proof}
By  \eqref{eq:equivalent_full_data} and the Proposition \ref{pr:topology} it holds that the map $\Psi$ is a well-defined homeomorphism. If $p \in \p M_1$, then by \eqref{eq:boundary_dist_agree} for any $z, w \in \p M_2$ we have
$$
(\D_2(\phi(p))(z,w)=d_2(\phi(p),z)-d_2(\phi(p),w)=d_1(p,\phi^{-1}(z))-d_1(p,\phi^{-1}(w))=((\Phi \circ \mathcal{D}_1)(p))(z,w).
$$
Applying $\D^{-1}_2$ for both sides of the equation above we have $\Psi(p)=\phi(p)$.
\end{proof}
\subsection{Smooth structure}
\label{Se:smooth}
In this part we show that the mapping $\Psi$ given in \eqref{eq:map_psi} is 
\\
a diffeomorphism. We consider separately the boundary and the interior cases. 

\medskip We start with the boundary case. Let $\sigma_{\p M}$ be the collection of all boundary cut points, 
$$
\sigma_{\p M}:=\{\gamma_{z,\nu}(\tau_{\p M}(z)) \in M: \: z \in M\}, \quad \tau_{\p M}(z):=\sup\{t>0:d_g(\p M, \gamma_{z,\nu}(t))=t\}.
$$
By Section III.4. of \cite{sakai1996riemannian} it holds that
\begin{equation}
\label{eq:boundary_cut_locus}
\sigma_{\p M}=\overline{\{p\in M: \#\{z\in\p M: d_g(p,z)=d_g(p,\p M)\}\geq 2\}}.
\end{equation}
Choose $w \in \p M$. Then by \eqref{eq:boundary_cut_locus} and the Proposition \ref{pr:topology} the data \eqref{eq:data_full} determine the set 
\begin{equation}
\label{eq:complement_of_boundary_cut_locus}
M \setminus \sigma_{\p M}=\{p \in M : \hbox{ The map $f_{p,w}$ has precicely one minimizer.}\}^{int},
\end{equation}
where $f_{p,w}$ is as in \eqref{eq:func_f_w}.
\begin{lemma}
\label{Le:equivalence_of_cut_loci}
Let $(M_i,g_i),\:  i=1,2$ be compact $n$--dimensional Riemannian manifolds with smooth boundaries $\p M_i$. If  the travel time difference data of $(M_1,g_1)$ and $(M_2,g_2)$ coincide, then
$$
M_2 \setminus \sigma_{\p M_2}=\Psi(M_1 \setminus \sigma_{\p M_1}).
$$
\end{lemma}
\begin{proof}
By the definition of the mapping $\Psi$ we have for any $p \in M_1$ and $w \in \p M_1$ that
$$
f^1_{p,w}(z)=f^2_{\Psi(p),\phi(w)}(\phi(z)), \quad z \in \p M_1,
$$
where $f^1_{p,w}$ and $f^2_{\Psi(p),\phi(w)}$ are defined as $f_{p,w}$ in  \eqref{eq:func_f_w}. Therefore the claim follows from \eqref{eq:complement_of_boundary_cut_locus}.
\end{proof}

Next we construct a boundary defining function on $M \setminus \sigma_{\p M}$. Let $p \in M \setminus \sigma_{\p M}$ and denote by $Z(p)$ the closest boundary point of $p$. The map $x \mapsto Z(x) \in \p M$ is smooth on $M \setminus \sigma_{\p M}$. Define a function
\begin{equation}
\label{eq:func_f_p}
f_p(z):=d_g(z,Z(p))-D_p(z,Z(p)), \quad z \in \p M.
\end{equation}
Notice that this function is determined by the data \eqref{eq:data_full}, and by triangular in equality the function $f_p$ is non-negative. If $p \in \p M$ then $f_p$ is a zero function. If $p \in M^{int} \setminus \sigma_{\p M}$ then  
\begin{equation}
\label{eq:func_f_p_outside_closes_bp}
f_p(z)>0, \quad z \in (\p M \setminus Z(p)).
\end{equation}
If this is not true then there exists $\p M \ni z \neq Z(p)$ such that
$$
d_g(p,z)=d_g(Z(p),z)+d_g(p,Z(p)).
$$ 
Which implies that there exists a distance minimizing curve from $p$ to $z$, that goes through $Z(p)$, but is not $C^1$ at $Z(p)$. By \cite{alexander1981geodesics} this is not possible. Thus \eqref{eq:func_f_p_outside_closes_bp} holds. Therefore we have proven the following
\begin{equation}
\label{eq:char_of_boundary}
\p M=\{p \in M \setminus \sigma_{\p M}: f_p \equiv 0\}.
\end{equation}

\begin{figure}
\begin{picture}(200,160)
\label{Fi:f_p}
  \put(-20,0){\includegraphics[height=12cm]{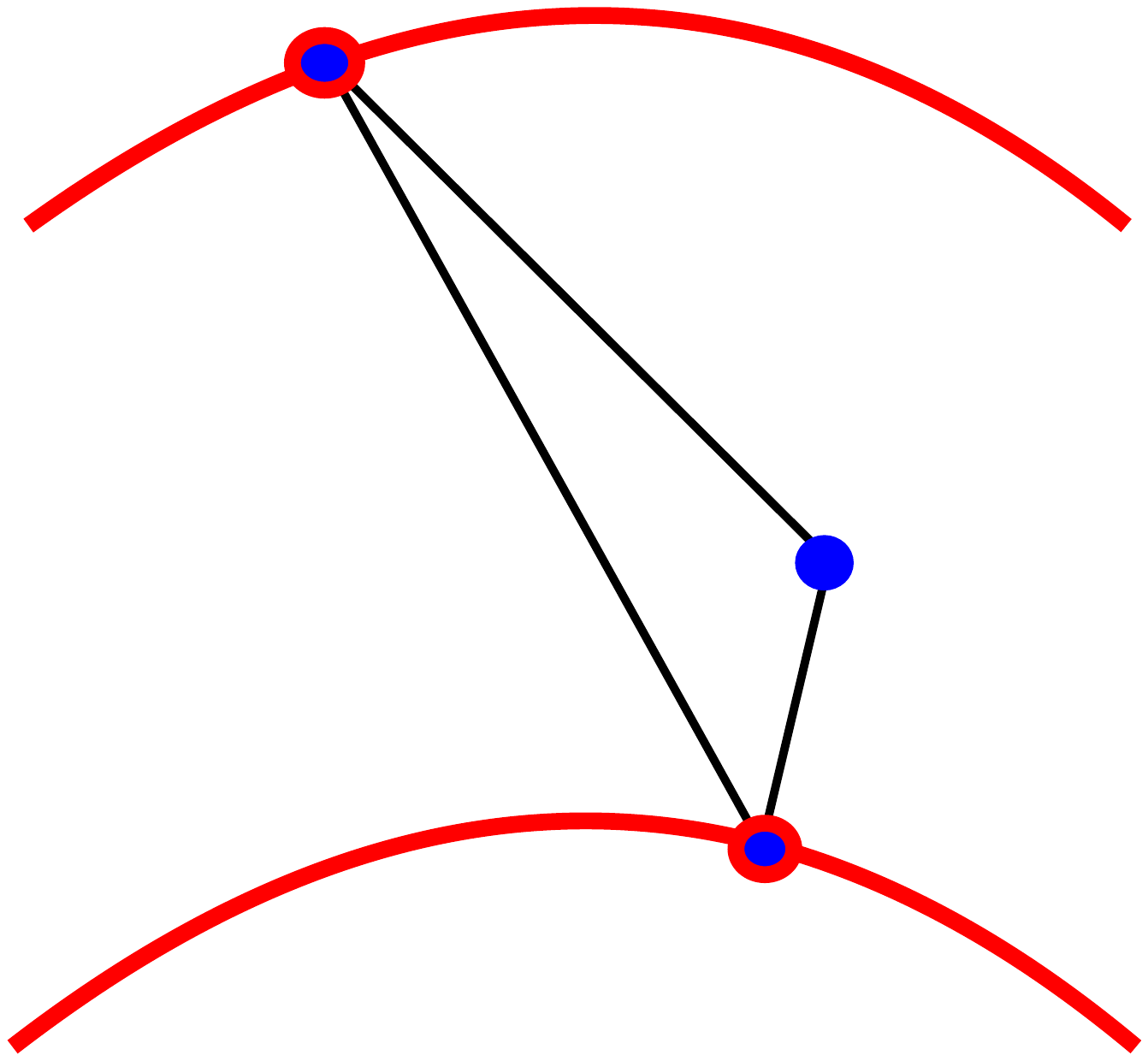}}
  \put(0,5){$\p M$}
   \put(115,15){$Z(p)$}
   \put(140,75){$p$}
    \put(0,115){$\p M$}
    \put(55,150){$z$}
   \end{picture}
  \caption{Here is the schematic picture of the  function $f_p$.}
\end{figure}
\begin{lemma}
\label{Le:smoothnes_of_dist_func}
Let $(M,g)$ be a smooth Riemannian manifold with smooth boundary for which the visibility condition \ref{eq:SU-cond-2} holds. Let $p \in \p M$. Then there exist $q \in \p M$ and neighborhoods $U,V \subset M$ of $p$ and $q$ respectively such that $d_g:U\times V$ is smooth. The distance minimizing geodesic from $p$ to $q$ is transversal to $\p M$ at $p$ and $q$. Moreover any distance minimizing geodesic $\gamma$ from $U$ to $V$ is contained $M^{int}$, if the start and end points are excluded.
\end{lemma}
\begin{proof}

We follow the proof of Theorem 1 of \cite{stefanov2009} and show that \ref{eq:SU-cond-2} implies the following claim: There exists $\eta \in S_pM, \hbox{ that is transversal to $\p M$ and }  0< \ell(p,\eta) < \infty,$ $ \gamma_{p,\eta}: [0,\ell(p,\eta)] \to N \hbox{ is distance}$ minimizer and $q:=\gamma_{p,\eta}(\ell(p,\eta))$ is not a cut point to $p$ along $\gamma_{p,\eta}$. The exit direction  $\dot{\gamma}_{p,\eta}(\ell(p,\eta))$  is transversal to $\p M$ and  $\gamma_{p,\eta}((0,\ell(p,\eta))) \subset M^{int}$. Moreover $\ell(p,\eta)=d_{g}(p,q)$. 

The claim of this lemma follows from implicit function theorem.


\end{proof}
Let $p \in \p M$. By Lemma \ref{Le:smoothnes_of_dist_func} there exists $w \in \p M$ and $r>0$ such that  the distance function $d_g$ is smooth in $B(p,r)\times B(w,r)$ and $B(p,r)\cap B(w,r)=\emptyset$. Let $r_{\p M}>0$ be the  minimum of $r$ and the boundary injectivity radius. Choose 
$$
z_0 \in (\p M \cap (B(w,r)) \hbox{ and }\delta \in (0,r_{\p M}),
$$
such that $z_0$ is not the closest boundary point  for any $q \in B(p,\delta)$, $Z(q) \in B(p,r) $ and the distance minimizing geodesic from $z_0$ to $p$ is not normal to $\p M$ at $p$. Then 
\begin{equation}
\label{eq:boundary_def_func}
E_{z_0}:B(p,\delta) \to [0,\infty), \quad E_{z_0}(q):=f_q(z_0)=d_g(z_0,Z(q))-D_q(z_0,Z(q))
\end{equation}
is well-defined and smooth. Moreover, by \eqref{eq:func_f_p_outside_closes_bp} we have that $E_{z_0}(q)=0$ if and only if $q \in B(p,\delta)\cap \p M$. Thus $E_{z_0}$ is a boundary defining function. Denote $(t,Z)$ for the boundary normal coordinates in $B(p,\delta)$, where $t(q)=d_g(\p M, q)$ and $Z(q)$ is the closest boundary point to $q \in B(p,\delta)$. Then the map 
\begin{equation}
\label{eq:boundary_coordinates}
W_{z_0}:B(p,\delta) \to [0,\infty) \times \p M, \quad W_{z_0}(q):=(E_{z_0}(q),Z(q)),
\end{equation} 
is smooth. 

We show that the Jacobian of this map with respect to boundary normal coordinates is invertible at $p$. By the inverse function theorem this yields the existence of a neighborhood $V\subset M$ of $p$ such that the restriction of $W_{z_0}$ to $V$ is a coordinate map. The Jacobian of $W_{z_0}$ at $p$ is
\begin{equation*}
\left(\begin{array}{cc}
\frac{\p}{\p t} E_{z_0}& \frac{\p}{\p t} Z
\\
\\
\frac{\p}{\p Z} E_{z_0}&\frac{\p}{\p Z}Z
\end{array}\right)
=
\left(\begin{array}{cc}
\frac{\p}{\p t} E_{z_0}&\bar 0^T
\\
\\
\frac{\p}{\p Z} E_{z_0}& Id_{n-1}.
\end{array}\right)
\end{equation*}
Notice
$$
\frac{\p}{\p t} E_{z_0}(t,Z)\bigg|_{(t,Z)=(0,p)}=1-g_p(\dot{\gamma}_{z_0,p}(d_g(p,z_0)),\nu)>0.
$$
The last inequlity holds since the distance minimizing geodesic $\gamma_{z_0,p}$ from $z_0$ to $p$ is not normal to the boundary at $p$. Thus  Jacobian of $W_{z_0}$ at $p$ is invertible.

\medskip

We use coordinates similar to $W_{z_0}$ to show that $\Psi:M_1\to M_2$ is a diffeomorphism near the boundary of $M_1$. In order to do so we first prove the following lemma.

\begin{lemma}
Let $(M_i,g_i),\:  i=1,2$ be compact $n$--dimensional Riemannian manifolds with smooth boundaries $\p M_i$. If  the travel time difference data of $(M_1,g_1)$ and $(M_2,g_2)$ coincide, then
\begin{equation}
\label{eq:boundary_metric}
g_1|_{\p M_1}=\phi^{\ast}(g_2|_{\p M_2}).
\end{equation} 
\end{lemma}
\begin{proof}
Since \eqref{eq:equivalent_data} implies \eqref{eq:boundary_dist_agree} the proof of this Lemma follows from the proof of Proposition 3.3. of \cite{zhou2012recovery}.  
%
%

\end{proof}

Now we are ready to prove the following lemma.
\color{black}

\begin{lemma}
\label{Le:boundary_coord}
Let $(M_i,g_i),\:  i=1,2$ be compact $n$--dimensional Riemannian manifolds with smooth boundaries $\p M_i$, whose travel time difference data coincide. Assume that $(M_1,g_1)$ satisfy the visibility condition \ref{de:TTDD_agree}. Let $p \in \p M_1$. There exists a neighborhood $U$ of $p$ in $M_1$ and $z_0 \in \p M_1$ such that on $U$ and $\Psi(U)$  the mappings $W^1_{z_0}(q_1)=(E^1_{z_0}(q_1),Z^1(q_1))$ and $W^2_{\phi(z_0)}(q_2)=(E^2_{\phi(z_0)}(q_2),Z^2(q_2))$ respectively, defined as in \eqref{eq:boundary_def_func} and \eqref{eq:boundary_coordinates},  are smooth local boundary coordinate maps. Moreover, with respect to these coordinates, the local representation of $\Psi$ is
\begin{equation}
\label{eq:local:rep_of_Psi_boundary}
W^1_{z_0}(U)\ni (s,z) \mapsto (s,\phi(z)) \in W^2_{\phi(z_0)}(\Psi(U)).
\end{equation}

\end{lemma}
\begin{proof}
By Lemma \ref{Le:equivalence_of_cut_loci} we have for any $q \in (M_1 \setminus \sigma_{\p M_1})$ that the point $z \in \p M_1$ is the closest boundary to $q$ if and  only if $\phi(z) \in \p M_2$ is the closest boundary point to $\Psi(q) \in (M_2 \setminus \sigma_{\p M_2})$. Thus 
$$
\phi(Z^1(q))=Z^2(\Psi(q)).
$$
Therefore, using  \eqref{eq:boundary_dist_agree} we have that for all $q \in (M_1 \setminus \sigma_{\p M_1}), \: z \in \p M_1$
\begin{equation}
\label{eq:f_p_agree}
f^1_q(z):=d_1(z,Z^1(q))-D_q(z,Z^1(q))=d_2(\phi(z),Z^2(\Psi(q))-D_{\Psi(q)}(\phi(z),Z^2(\Psi(q))=:f^2_{\Psi(q)}(\phi(z)).
\end{equation}

We choose $w \in \p M_1$ neighborhoods $U'$ and $V$ for $p$ and $w$ respectively as in Lemma \ref{Le:smoothnes_of_dist_func} for $(M_1,g_1)$. Then function $(x,z) \mapsto d_1(x,z)$ is smooth in $(U' \cap \p M_1) \times V\cap \p M_1)$. Let $\gamma$ be the unique distance minimizing geodesic from $p$ to $w$ that is transversal to $\p M_1$ at $p$ and $w$. Since $\phi$ is a diffeomorphism by \eqref{eq:boundary_dist_agree} and \eqref{eq:boundary_metric} it follows that
\[
D\phi\;\bigg( \hbox{grad}'_1\;d_1(\cdot,w)\bigg|_{p}\bigg)= \hbox{grad}'_2\;d_2(\cdot,\phi(w))\bigg|_{\phi(p)} .
\]
Here $\hbox{grad}'_i$, $i\in \{1,2\}$ stands for the boundary gradient. Therefore, a $g_2$--distance minimizing unit speed curve $c$ from $\phi(p)$ to $\phi(w)$ is transversal to $\p M_2$ at $\phi(p)$. Switching the order of $p$ and $w$ we prove also that $c$ is transversal to $\p M_2$ at $\phi(w)$. Since $\gamma$ is the unique distance minimizing curve from $p$ to $w$ and $\gamma((0,d_1(p,w))) \subset M_1^{int}$ it holds by \eqref{eq:boundary_dist_agree} that  $c((0,d_2(\phi(p),\phi(w)))) \subset M_2^{int}$. Therefore $c$ is a geodesic of $g_2$. Since $d_2(\phi(p),\cdot)|_{\p M}$ is smooth at $\phi(w)$, $c$ is the unique distance minimizing curve of $(M_2,g_2)$ connecting $\phi(p)$ to $\phi(w)$. Moreover due to transversality of $c$ there exists a neighborhood of $\phi(w)$ such that any point in this neighborhood is connected to $\phi(p)$ via the unique distance minimizing geodesic. Since conjugate points of $\phi(p)$ in $(M_2,g_2)$ are accumulation points of those points $q\in M_2$ that can be connected to $\phi(p)$ via multiple distance minimizers, it holds that $\phi(w)$ is not either a conjugate point of $\phi(p)$ along $c$. Therefore $\phi(w)$ is not a cut point of $\phi(p)$ along $c$. This proves that also $(M_2,g_2)$ satisfies the visibility condition.

By Lemma \ref{Le:smoothnes_of_dist_func} we have proved that there exists $r_{\min}>0$ smaller than the minimum of the boundary cut distances of $g_1$ and $g_2$, such that functions 
\[
(q,z) \mapsto d_1(q,Z^1(q)), \:d_1(q,z), \:d_1(z,Z^1(q)), \quad (q,z) \in B_1(p,r_{\min}) \times (B_1(w,r_{\min})\cap \p M_1)
\]
and
\[
(q',z') \mapsto d_2(q',Z^2(q')), \:d_2(q',z'),\: d_2(z',Z^2(q)), \quad (q',z') \in B_2(\phi(p),r_{\min}) \times (B_2(\phi(w),r_{\min})\cap \p M_2)
\]
are smooth. Since $\Psi$ is a homeomorphism the existence of set $U$ and $z_0 \in \p M_1$ as in the claim of this Lemma follow.

If $q \in U$ we obtain by \eqref{eq:f_p_agree} the following equation
$$
E^1_{z_0}(q)=E^2_{\phi(z_0)}(\Psi(q)).
$$
Therefore we have proven that the map given in \eqref{eq:local:rep_of_Psi_boundary} and the mapping
$$
W^2_{\phi(z_0)}\circ \Psi\circ (W^1_{z_0})^{-1}:W^1_{z_0}(U) \to W^2_{\phi(z_0)}(\Psi(U))
$$ 
coincide.

\end{proof}
\color{black}

Next we consider the coordinates away from $\p M$. Let $p \in M^{int}$ and choose any closest boundary point $z_p\in \p M$ to $p$. By Lemma 2.15 of \cite{Katchalov2001} there exist neighborhoods $U \subset M^{int}$ of $p$ and $W \subset \p M$ of $z_p$ such that the distance function $d_g:U\times W \to \R$ is smooth. Moreover for every $(q,w) \in U\times W$ the distance $d_g(q,w)$ is realized by the unique distance minimizing geodesic, contained in $M^{int}$, if the end point $w$ is excluded. We use a shorthand notation $v \in S_p M$ for the velocity $\dot \gamma_{z_p,\nu}(d_g(p,z_p))$.  A similar argument as in Lemma 2.6. of \cite{LaSa} yields to an existence of a neighborhood $V \subset W$ of $z_p$ such that the set  
$$
\mathcal{V}=\{(z_i)_{i=1}^n \in V^n: \dim \hbox{span}((F(z_i)-v)_{i=1}^n)=n\} 
$$
is open and dense in $V^n:= V\times V \times \ldots \times V$. Here $F(q):=-\frac{(\exp_p)^{-1}(q)}{\|(\exp_p)^{-1}(q)\|_g}, \: q \in V$.  Notice that this claims follows from Lemma 2.6. of \cite{LaSa} since $F(q)=\frac{(\exp_p)^{-1}(q')}{\|(\exp_p)^{-1}(q')\|_g}$ for some $q'\in M$ if and only if there exists $ 0<t<\tau(p,-F(q))$ such that $q'=\gamma_{p,-F(p)}(t)$.

Moreover for every $(z_i)_{i=1}^n \in \mathcal{V}$ there exists an open neighborhood $U' \subset U$ of $p$ such that 
$$
H:U'\rightarrow \R^n, \quad H(q)=(d_g(q,z_i)-d_g(q,z_p))_{i=1}^n
$$
is a smooth coordinate mapping. This holds, since for any $(z_i)_{i=1}^n \in \mathcal{V}$ the Jacobian of $H$ at $p$ is invertible. 

\begin{lemma}
\label{Le:interior_coord}
Let $(M_i,g_i),\:  i=1,2$ be compact $n$--dimensional Riemannian manifolds with smooth boundaries $\p M_i$. Suppose that the travel time difference data of $(M_1,g_1)$ and $(M_2,g_2)$ coincide. Let $p \in M_1^{int}$. Let $z_p$ be any closest boundary point to $p$. There exists a neighborhood $U$ of $p$ in $M_1^{int}$ and a neighborhood $W \subset \p M_1$ of $z_p$ such that the distance functions $d_1:U \times W$ of $(M_1,g_1)$ and $d_2:\Psi(U) \times \phi(W)$ of $(M_2,g_2)$ are smooth. 

Moreover there exists points $z_1, \ldots, z_n \in W$ and a neighborhood $V \subset U$ of $p$ such that  
$$
H_1:V\rightarrow \R^n, \quad H_1(x)=(d_1(x,z_i)-d_1(x,z_p))_{i=1}^n
$$
and
$$
H_2:\Psi(V)\rightarrow \R^n, \quad H_2(q)=(d_2(q,\phi(z_i))-d_2(q,\phi(z_p)))_{i=1}^n,
$$
are smooth coordinate maps. We also have 
\begin{equation}
\label{eq:local:rep_of_Psi_interior}
H_1(V)= H_2(\Psi(V)) \hbox{ and } H_2 \circ \Psi\circ H_1 = Id_{\R^n}.
\end{equation}

\end{lemma}
\begin{proof}
Since $\Psi$ is a homeomorphism, the first part of the claim follows from similar construction as done before this Lemma. The proof of the latter part is a modification of the proof of Theorem 2.7. of \cite{LaSa}. 
\end{proof}

\begin{proposition}
\label{th:diffeo}
Let $(M_i,g_i),\:  i=1,2$ be compact $n$--dimensional Riemannian manifolds with smooth boundaries $\p M_i$ whose travel time difference data coincide. If $(M_1,g_1)$ satisfy the visibility condition \ref{eq:SU-cond-2}, then mapping $\Psi:M_1 \to M_2$, given in \eqref{eq:map_psi}, is a diffeomorphism.
\end{proposition}
\begin{proof}
The claim follows from Proposition \ref{th:topology} and lemmas \ref{Le:boundary_coord}--\ref{Le:interior_coord}. 
\end{proof}

\subsection{Riemannian structure}
\label{Se:Riemannian}
As we have proven that the map $\Psi$ is diffeomorphism we can define a pull back metric $\widetilde g:=\Psi^\ast g_2$ on $M_1$. From now on we only consider manifold $M:=M_1$ with smooth boundary equipped with Riemannian metrics $g:=g_1$ and $\widetilde g$. We need  to show that $g=\widetilde g$.
First we notice that by the definitions of the diffeomorphism $\Psi$ and metric $\widetilde g$ on $M$ we have by the data \eqref{eq:data_full} that
\begin{equation}
\label{eq:dist_dif_agree}
D_p(z,w)=d_g(p,z)-d_g(p,w)=d_{\widetilde g}(p,z)-d_{\widetilde g}(p,w), \quad p \in M, \: z,w \in \p M.
\end{equation}

\begin{lemma}
\label{Le:jet}
Let $p \in \p M$ and $(x^1,\ldots, x^n)$ be a boundary normal coordinate system of $g$ near $p$ and $\alpha\in \N^{n}$ any multi-index. Write $g=(g_{ij})_{i,j=1}^{n}$ and $\widetilde g=(\widetilde g_{ij})_{i,j=1}^{n}$. Then for all $i,j \in \{1, \ldots,n\}$ holds
\begin{equation}
\label{eq:jet}
\partial^\alpha g_{ij} |_{\p M}= \partial^\alpha \widetilde g_{ij} |_{\p M}, \quad \p^\alpha:=\prod_{k=1}^n \left(\frac{\p}{\p x^k}\right)^{\alpha_k}.
\end{equation}
\end{lemma}
\begin{proof}
We prove that the local lens relations $(\ell_g,\sigma_g)$ and $(\ell_{\widetilde g},\sigma_{\widetilde g})$ of $g$ and $\widetilde g$ respectively coincide at some open set $\mathcal D \subset T\p M$. After this the claim follows from the proof of Theorem 1 of \cite{stefanov2009}. For the definitions of local lens relations see \cite{stefanov2009}.

Choose $q \in \p M$ and neighborhoods $U,V \subset M$ of $p$ and $q$ be as in Lemma \ref{Le:smoothnes_of_dist_func} for metric $g$. Let $\gamma$ be the unique geodesic of $g$  connecting $p$ to $q$.
Due to \eqref{eq:boundary_dist_agree}
and Lemma \ref{Le:smoothnes_of_dist_func} it holds that $d_{\widetilde g}$ is smooth on $(U \cap \p M) \times (V \cap \p M)$.  Therefore we have for every $(x,y) \in (U \cap \p M) \times (V \cap \p M)$ that
\begin{equation}
\label{eq:boundary_gradients}
\hbox{grad}'_g \; d_g(\cdot,y)\bigg|_{x}=\hbox{grad}'_{\widetilde g} \; d_{\widetilde g}(\cdot,y)\bigg|_{x} \quad \hbox{ and } \quad  \hbox{grad}'_g \; d_g(\cdot,x)\bigg|_{y}=\hbox{grad}'_{\widetilde g} \; d_{\widetilde g}(\cdot,x)\bigg|_{y}.
\end{equation}

Denote $\dot{\gamma}(0)=:\eta$ and $\dot{\gamma}(d_g(p,q))=:v$. Then \eqref{eq:boundary_metric} and \eqref{eq:boundary_gradients} imply that $\dot{\widetilde \gamma}(0)=\eta$ and $\dot{\widetilde \gamma}(d_g(p,q))=v$, where $\widetilde \gamma$ is the unique distance minimizing geodesic of $\widetilde g$ from $p$ to $q$. By Lemma \ref{Le:smoothnes_of_dist_func} it holds that $\eta$ and $v$ are transversal to $\p M$. 

Therefore after  possibly shrinking $U$ and $V$ we have by formula (10) of \cite{stefanov2009} and formulas \eqref{eq:boundary_metric} and \eqref{eq:boundary_gradients} that the  local lens relations $(\ell_g,\sigma_g)$ and $(\ell_g,\sigma_{\widetilde g})$ coincide in the set
\[
\mathcal{D}:=\{ \hbox{grad}'_g \;d_g(\cdot,y)\bigg|_{x}, \: \hbox{grad}'_g \; d_g(\cdot,x)\bigg|_{y} \in T\p M: \: (x,y) \in (U \cap \p M) \times (V \cap \p M)\}.
\]
The set $\mathcal D$ is open since it is an image of an open map, given by the composition of the  diffeomorphism
\[
 W_\eta\ni (x,v) \mapsto \gamma_{x,v}(\ell(x,v)),\dot \gamma_{x,v}(\ell(x,v)) \in W_v 
\]
and the orthogonal projection from $\p SM$ to $T\p M$. In the above $W_\eta\subset \p SM$ is some open neighborhood of $(p,\eta)$ and $ W_v\subset \p SM$ is some open neighborhood of $(q,v)$.  
\end{proof}
\color{black}

Let $(N,G)$ be a smooth closed Riemannian manifold that is a smooth extension of $(M,g)$. We write $F:= N \setminus M^{int}$, as before. By Lemma \ref{Le:jet} $(N,\widetilde G)$ is a smooth extension of $(M,\widetilde g)$, if $\widetilde G$ is a Riemannian metric defined as 
\begin{equation}
\label{eq:def_G_tilde}
\widetilde G|_F=G|_F, \quad \widetilde G|_{M^{int}}=\widetilde g.
\end{equation}
\begin{lemma}
Let $N, F, G$ and $\widetilde G$ be as above. Then 
\begin{equation}
\label{eq:dist_dif_on_N}
d_G(p,z)-d_G(p,w)=d_{\widetilde G}(p,z)-d_{\widetilde G}(p,w) \quad p \in N, \: z,w \in F.
\end{equation}
The functions $d_G, d_{\widetilde G}$ are the geodesic distances of $G$ and $\widetilde G$ respectively.
\end{lemma}
\begin{proof}
If $p \in M$, we will soon give  a proof for
\begin{equation}
\label{eq:dist_dif_on_N_2}
d_G(p,z)-d_G(p,w)=d_{\widetilde G}(p,z)-d_{\widetilde G}(p,w), \quad z,w \in F.
\end{equation}  
This proof is an adaptation of Proposition 7.3 in \cite{ivanov2018distance}. If \eqref{eq:dist_dif_on_N_2} holds for every $p\in M$ then   \eqref{eq:dist_dif_on_N_2} holds also for the case $p \in F$. The latter proof is given in Proposition 1.2. of \cite{LaSa}. Therefore  equation \eqref{eq:dist_dif_on_N} holds.

Let $p \in M$. Consider first the function $h_p(z)=d_{g}(p,z)-d_{\widetilde g}(p,z), \: z \in \p M$. 
Let $w \in \p M$. By \eqref{eq:dist_dif_agree} it holds that
\[
h_p(z)=d_{g}(p,w)-d_{\widetilde g}(p,w). 
\] 
Thus $h_p$ is a constant function.

\medskip We will prove that

\begin{equation}
\label{eq:proof_of_(31)}
d_G(p,z)=\inf\bigg\{d_g(p,y_0)+ \bigg(\sum_{j=1}^Nd_F(y_{j-1},x_j)+d_g(x_{j},y_j)\bigg)+d_F(x_{N},z)\bigg\},
\end{equation}
where $d_F$ is the distance function of the Riemannian manifold $(F,G|_F)$ and
\\
$\{y_0, \ldots, y_N, x_1, \ldots, x_N\} \subset \p M$. Notice that similar formula holds for $d_{\widetilde G}$, when $d_g$ is replaced with $d_{\widetilde g}$. If \eqref{eq:proof_of_(31)} holds then, it follows from equation \eqref{eq:boundary_dist_agree} that
\[
d_G(p,z)-d_{\widetilde G}(p,z)=h_p(z)=\hbox{constant with respect to $z$}.
\]
This implies \eqref{eq:dist_dif_on_N}, in the case when $p \in M$.

Finally we will prove \eqref{eq:proof_of_(31)}. Let $\epsilon>0$. Since $\p M$ is a smooth co-dimension 1 submanifold of $N$, it follows from the definition of the Riemannian distance function $d_G$, that there exists a piecewise smooth curve $c$ from $p$ to $q$, that crosses the boundary finitely many times, and whose length is $\epsilon$--close to $d_G(p,z)$. Then 
\[
d_g(p,y_0)+ \bigg(\sum_{j=1}^Nd_F(y_{j-1},x_j)+d_g(x_{j},y_j)\bigg)+d_F(x_{N},z)\leq \mathcal{L}_G(c)\leq d_G(p,z)+\epsilon,
\] 
where $\{y_0, \ldots, y_N, x_1, \ldots, x_N\} \subset \p M$ are the points where $c$ crosses the boundary. Taking $\epsilon$ to $0$ implies \eqref{eq:proof_of_(31)}.
%
%

\end{proof}

Due to the previous Lemma it follows from the Section 2.4 of  \cite{LaSa} that metric tensors $G$ and $\widetilde G$ coincide. We will sketch here the main ideas for this proof.

\medskip
First we prove that the geodesics of metrics $G$ and $\widetilde G$ agree up to reparametrization. Let $\tau_G:SN \to \R$ be the cut distance function of metric tensor $G$. By Lemma 2.9. of \cite{LaSa} the following equality holds for any $(z,v)\in SF^{int}$ 
\begin{equation}
\label{eq:image_of_geo_1}
\gamma^G_{z,-v}((0,\tau_G(z,-v))=\{p \in N:D_p(\cdot,z) \hbox{ is smooth at $z$ and grad$_G D_p(\cdot,z)$ at $z$ is $v$} \}. 
\end{equation}
Where $\gamma^G_{z,-v}$ is the geodesic of $G$ with initial conditions $(z,-v)$. Since $G=\widetilde G$ on $F^{int}$, the formulas  \eqref{eq:dist_dif_on_N} and \eqref{eq:image_of_geo_1} imply
\begin{equation}
\label{eq:image_of_geo_2}
\gamma^G_{z,-v}((0,\tau_G(z,-v))=\gamma^{\widetilde G}_{z,-v}((0,\tau_{\widetilde G}(z,-v)), \quad (z,v) \in SF^{int},
\end{equation}
where $\tau_{\widetilde G}$ is the cut distance function of $\widetilde G$. Therefore, for any $(z,v) \in SF^{int}$ there exists a diffeomorphism $\alpha_{z,v}:(0,\tau_G(z,-v)) \to (0,\tau_{\widetilde G}(z,-v))$ such that
\begin{equation}
\label{eq:image_of_geo_3}
\gamma^G_{z,-v}(t)=\gamma^{\widetilde G}_{z,-v}(\alpha_{z,v}(t)), \quad t \in (0,\tau_G(z,-v)), \:
 (z,v) \in SF^{int}.
\end{equation}

Let $p \in M^{int}$. We denote the exponential map of $G$ at $p$ by $\exp_p$. Then the following set is not empty,
$$
\Omega_p:=\{rv\in T_pN: r>0, \:  v=\exp_p^{-1}(z), \: p \in \sigma(z,v), \: (z,v) \in SF^{int}\}^{int},
$$
and, moreover, if we denote the exponential map of $\widetilde G$ at $p$ by $\widetilde \exp_p$. In view of  \eqref{eq:image_of_geo_3} we have
\begin{equation}
\label{eq:geodesic_in_M}
\Omega_p=\{rv\in T_pN: r>0, \:  v=\widetilde \exp_p^{-1}(z), \: p \in \sigma(z,v), \: (z,v) \in SF^{int}\}^{int}.
\end{equation}

Let $(U,x)$ be a local coordindate chart of $M^{int}$. We denote the Christoffel symbols of $G$ and $\widetilde G$ as $\Gamma$ and $\widetilde \Gamma$, respectively.  By \eqref{eq:image_of_geo_3}, \eqref{eq:geodesic_in_M} 
and Proposition 2.13 of \cite{LaSa} there exists a smooth $1$--form $\beta$ on $U$ such that
$$
\Gamma^k_{ij}(x)-\widetilde \Gamma^k_{ij}(x)=\delta^k_i\beta_j(x)+\delta^k_j\beta_i(x),
$$
where $\delta^k_j$ is the Kronecker delta. This and Lemma 2.14 of \cite{LaSa} imply that the geodesics of metric tensors $G$ and $\widetilde G$ agree up to reparametrization. See also \cite{matveev2012geodesically} for the similar result. We arrive at.

\begin{lemma}
Suppose that $N, F, G$ and $\widetilde G$ are as above. Then $G=\widetilde{G}$ in all of $N$.
\label{Le:geodesic eq -> metrics are the same}
\end{lemma}
\begin{proof}
Since geodesics of metric tensors $G$ and $\widetilde G$ agree up to reparametrization the main result of \cite{topalov2003geodesic} shows that the function
\begin{equation}
\label{Matveev formula}
I_0((x,v))=\bigg(\frac{\det (G(x)) }{\det(\widetilde{G}(x))}\bigg)^{\frac{2}{n+1}} \widetilde{G}(x,v), \quad (x,v) \in TN,
\end{equation}
where  $\widetilde {G}(x,v)=\widetilde{G}_{jk}(x)v^jv^k$, is constant on the geodesic flow of $G$. 
Note that the function $F(x):= \frac{\text{det} (G(x)) }{\text{det}(\widetilde{G}(x))}$ is coordinate invariant. 

Let $\varphi_t:SN \to SN$, $t\in \R$ be the geodesic flow of $G$ and $\pi:TN \to N$ the projection onto the base point. 
Since $G=\widetilde G$ on $F^{int}$, we have
$$
G(\varphi_0(z,v)) =\|v\|^2_G=I_0(\varphi_0(z,v), \quad (z,v) \in TF^{int} .
$$
Therefore for any $t \in \R$ and for any $(z, v) \in TF^{int}  \setminus \{0\}$ the following holds
$$
G(\varphi_t(z,v)) =\|v\|^2_G=I_0(\varphi_t(z,v)=F(\pi(\varphi_t(z,v))\widetilde G(\varphi_t(z,v)).
$$
This implies the claim. For more details, see Lemma 2.15 of \cite{LaSa}.
\end{proof}

We conclude that the proof of Theorem \ref{th:main} follows from Propositions \ref{th:topology}, \ref{th:diffeo} and Lemma \ref{Le:geodesic eq -> metrics are the same}.

\subsection*{Acknowledgements}
We would like to express our gratitude for Matti Lassas and Sergei Ivanov for the excellent suggesstion during the preparation of this paper. M. V. dH. was partially supported by the Simons Foundation under the MATH + X program, the National Science Foundation under grant DMS-1559587, and by members of the Geo-Mathematical Imaging Group at Rice University. TS was supported by the Simons Foundation under the MATH + X program. 
\color{black}

\bibliographystyle{abbrv} 
\bibliography{bibliography}
\end{document}